\documentclass[twoside,12pt,leqno]{amsproc}

\usepackage{amssymb,latexsym,enumerate,verbatim,stmaryrd,tikz,mathtools}
\usepackage[pagebackref]{hyperref}
\usepackage{amsrefs,wasysym}    
\usepackage{etoolbox}   
\usepackage{booktabs}
\hypersetup{citecolor=red, linkcolor=blue, colorlinks=true}
\numberwithin{table}{section}

\newtheorem{theorem}{Theorem}[section]
\newtheorem{lemma}[theorem]{Lemma}

\newfont{\mate}{msbm10}
\newfont{\matt}{msam10}

\newcommand{\GAP}{{\sf GAP}}

\newcommand{\F}[1]{\hbox{\mate F}$\!_{#1}$}
\newcommand{\Z}{{\mate Z}}
\newcommand{\Fp}{\hbox{{\mate F}$\!_p$}}
\newcommand{\A}{{\mathcal A}}
\renewcommand{\P}{{\mathcal P}}
\newcommand{\R}{{\mathcal R}}
\newcommand{\X}{{\mathcal X}}

\renewcommand{\L}{{\mathcal L}}

\newcommand{\epim}{{\matt \symbol{16}}}
\newcommand{\pcpG}{\{ \A \mid \R \}}
\newcommand{\pcphG}{\{ \hat{\A} \mid \hat{\R} \}}
\newcommand{\pcptG}{\{ \tilde{\A} \mid \tilde{\R} \}}

\begin{document}

\title{A Finite Soluble Quotient Algorithm}
\author[Alice C.\ Niemeyer]{Alice C.\ Niemeyer}
\address{        Department of Mathematics\\
        University of Western Australia,
        Nedlands, WA 6009,
        Australia}

\maketitle

\begin{abstract}
An algorithm for  computing  power conjugate presentations for  finite
soluble  quotients  of predetermined structure  of finitely  presented
groups  is described.   Practical aspects  of  an  implementation  are
discussed. 
\end{abstract}

\section{Introduction}
 
Polycyclic  groups are  characterised  by the fact  that  they  have a
polycyclic  series,  which  is a descending series of subgroups,  such
that  each one  is  normal in the previous  one and their quotient  is
cyclic  (see Segal, 1983, or Sims,  1994).   Polycyclic  presentations
describe polycyclic groups by exhibiting  a polycyclic series  of  the
group.  They are very important  for computing in the group since they
allow the practical computation (by  collection) of a normal  word for
every element in the group.  Algorithms  using  such  descriptions for
finite polycyclic groups are, for  example,  described in Laue {\it et
al.}  (1984) and form  an  integral  part of  the  computational group
theory systems {\sc Magma} (see Bosma and Cannon,  1993) and {\sf GAP}
(Sch\"onert {\it et al.,} 1994).
 
Every polycyclic  group is soluble and  every finite  soluble group is
polycyclic. However, not every  soluble group is polycyclic.  Baumslag,
Cannonito, and  Miller  (1981a,  1981b)  describe an  algorithm  which
decides  whether a soluble  group  given by  a finite presentation  is
polycyclic.  It   has been partly implemented  by  Sims  (1990).  Here
attention is   focused   on  finite   soluble   groups.     Polycyclic
presentations for finite   soluble  groups are better known   as power
conjugate presentations.
 
The task  of a finite soluble  quotient algorithm is to  compute power
conjugate presentations   for  finite   soluble groups   described  as
quotients of finitely presented   groups.  A number of proposals   for
finite  soluble  quotient algorithms have  been  made, for instance by
Wamsley (1977),  by Leedham-Green (1984) and by  Plesken  (1987).  The
last algorithm has been  implemented by Wegner (1992).
 
Algorithms for computing power conjugate presentations for  $p$-groups
or for nilpotent groups described  as quotients  of finitely presented
groups exist. For $p$-groups see for example Havas and  Newman  (1980)
or Celler {\it  et al.\  }(1993)  and for nilpotent  groups  see  Sims
(1994) or Nickel (submitted).
 
Here a  new finite soluble quotient algorithm is  presented in detail.
For a brief description  of the  algorithm  see Niemeyer (to  appear).
New  features  are the use of vector enumeration and the  intermediate
presentations considered. Leedham-Green (1984) suggested the  use of a
vector enumerator in this context.
 
We  now turn  to  the background required for  the  description of the
algorithm. Let $G$ be a finite soluble group and let  $G = G_0 \ge G_1
\ge \cdots \ge  G_n = \{ 1 \}$  be  a composition series for $G$  with
factors of prime order.  Choose elements $a_i \in G$  for $1 \le i \le
n$ such that  $G_{i-1} = \langle G_i, a_i  \rangle;$ let $p_i$ be  the
order of the  factor $G_{i-1}/G_i.$ Then $\A = \{ a_1,\ldots, a_n  \}$
is a  generating set for  $G.$  Choose words $v_{jk}$ in the  elements
$a_{j+1}, \ldots, a_n$ for $1 \le j \le k \le n$ such that $a_i^{p_i}=
v_{ii}$ for $1 \le i \le n$ and  $a_k^{a_j}= v_{jk}$ for $ 1\le j<k\le
n$ and let $\R$ be the set consisting of these  relations.   Then $\R$
is a defining set of relations for $G.$ The  presentation $\pcpG$ is a
{\it  power  conjugate  presentation}  for  $G.$  A   power  conjugate
presentation $\pcpG$ of  a group $G$  {\it  exhibits} the  composition
series $G = G_0 \ge G_1 \ge \cdots \ge G_n = \langle 1 \rangle,$ where
$G_{i-1} = \langle a_i,\ldots,  a_n \rangle$ for  $1 \le i \le n.$ The
order  of  $G$  is  at  most $\prod_{i=1}^n p_i$ and therefore $G$  is
finite.  A  word $w(a_1, \ldots,  a_n)$  in  the  generators  is  {\it
normal\/}  if it is of the form $a_1^{e_1}\cdots a_n^{e_n}$ with $0\le
e_i < p_i.$ Note that we only consider words  in the elements of $\A,$
that is these words do not contain inverses of the generators.
 
In what follows ``word'' means semigroup word.  A  normal  word $u$ in
$\A$ is {\it equivalent} to a word $w$ in $\A$ if  $u$ and $w$ are the
same  element  of  the  group  defined  by  $\pcpG.$  The  fundamental
importance  of   power  conjugate   presentations   arises  from   the
observation that, given  a word $w$ in  the generators $\A,$ the power
conjugate presentation $\pcpG$  can be used  to compute an  equivalent
normal  word.   It  is assumed  that  the  words  $v_{jk}$ in  a power
conjugate presentation  are normal.  If  the right hand  side of  some
relation is the identity then the relation is written as a relator  by
just listing the left hand side.
 
If a  word is not normal it  has a non-normal  subword minimal  in the
partially ordered set of non-normal words. This subword is of the form
$a_i^{p}$ or of the form $a_k  a_j$ with $k > j$.  {\it Collection} of
a word consists of a sequence of steps each of which chooses a minimal
non-normal subword and replaces it.  A subword of the form $a_i^p$  is
replaced  by $v_{ii}$ and $a_k {a_j}$ is replaced by $a_j v_{jk}.$ For
normal words the sequence of collection steps is  empty.  Otherwise  a
minimal non-normal subword is  chosen and replaced.  Each further step
is  applied  to  the result of the previous step.  The words resulting
from these steps are equivalent to  the given word.  For an account of
collection   see   for  example  Havas   and   Nicholson   (1976),  or
Leedham-Green and Soicher (1990).

A  fundamental  result is that  every  collection (independent  of the
choices of minimal non-normal subwords) of a non-normal  word  results
in a normal word after a finite number of steps (see for example Sims,
1994).  A  normal word resulting from collecting the word $w$ will  be
denoted ($w$);  it may depend upon  the choices  made in  the process.
Multiplication  of  two elements of $G$ amounts to computing  a normal
word for the product given by concatenation.

In  general, there  may  exist many normal words representing  a given
group  element.  If each  element is represented  by a  unique  normal
word, then the power  conjugate presentation  is {\it consistent}.  In
this case two group elements are equal only if they are represented by
the  same normal  word.  For  the  finite soluble  group $G,$ given as
above, the order is then equal to $\prod_{i=1}^n p_i.$

The following  result is due to  Wamsley (1977).  In summary it states
that a power  conjugate  presentation is  consistent if certain words,
called  {\it   consistency  test   words},   can   be   collected   in
``sufficiently different'' ways and still yield the same normal word.
These consistency test words are $ \vphantom{a_k^{p-1}}a_k^{}\,
a_j^{} a_i^{}$ with $ 1 \le i < j < k \le n,$
$a_k^p a_j^{},$ with $1 \le j < k \le n,$ 
$a_j a_i^{p},$ with $1 \le i < j \le n,$ and
$ a_i^{p+1}$ with $1 \le i \le n.$

\begin{theorem}
\label{consis}
Let $G$ be a finite soluble group given by the power conjugate
presentation $\pcpG.$
Then the presentation $\pcpG$   is consistent if and only if the following
equations hold: 
$$\vbox{\openup 3pt
\halign{\hfil$#$\hfil &\quad{\it #}\hfil& \hskip .5em$#$\hfil\cr
\bigl(\vphantom{a_k^{p-1}}(a_k^{}\, a_j^{})\, a_i^{}\bigr) \, 
= \bigl( \vphantom{a_k^{p-1}}a_k^{}\,(a_j^{}\, a_i^{})\bigr) & for
& 1 \le i < j < k \le n,\cr
\bigl(\vphantom{a_k^{p-1}}(a_k^p)\, a_j^{} \bigr)\, = \bigl( a_k^{p-1}\,
(a_k^{}\, a_j^{})\bigr) & for 
&1 \le j < k \le n,\cr
\bigl(\vphantom{a_k^{p-1}}(a_j^{}\,a_i^{})\,a_i^{p-1} \bigr)\,
= \bigl(\vphantom{a_k^{p-1}} a_j^{}\,(a_i^p)\bigr) & for 
&1 \le i < j \le n,\cr 
\bigl(\vphantom{a_k^{p-1}} (a_i^p)\, a_i^{}\bigr)\, = \bigl( 
\vphantom{a_k^{p-1}}a_i^{}\, 
(a_i^p)\bigr) & for &  1 \le i \le n.\cr}}$$
\end{theorem}

In practice we are interested in describing groups by consistent power
conjugate   presentations.    In    this   context   power   conjugate
presentations which  exhibit  a  refinement of a specific series  of a
given  group are  considered.   The  algorithm described  here can  be
viewed as a generalisation of the  $p$-quotient algorithm described by
Havas and Newman (1980)  and by Celler  {\it et al.\  }(1993). It also
has features in common with the nilpotent quotient algorithm described
by Nickel (submitted).  Let $G$ be a group and $p$ a prime.   The
prime quotient algorithm works with a series
$$G   =   {\P}^p_0(G)\ge {\P}^p_1(G)\ge\cdots\,\,\hbox  to  1.5   cm
{\hfil with\  }  {\P}^p_i(G)  = [{\P}^p_{i-1}(G),G] \bigl({\P}^p_{i-1}
(G)\bigr)^p  {\rm\ for\ } i \ge 1$$ called 
the {\it lower exponent-$p$ central series} of $G.$ If there exists an
integer $c\ge 0$ such that ${\P}^p_c(G)=\langle 1\rangle,$ then $G$ is
a  $p$-group  and  the  smallest  such  integer  is  called  the  {\it
exponent-$p$  class}  of  $G.$ It repeats a basic step which, given  a
consistent   power  conjugate  presentation  of  $G/{\mathcal  P}_{i}(G),$
computes  a  consistent  power  conjugate  presentation   of  $G/{\mathcal
P}_{i+1}(G).$

In the  next section a series  is defined  that takes the role of  the
lower  exponent-$p$ central series in  the context of  finite  soluble
groups. The  power conjugate presentations exhibiting this  series are
described.

\section{The soluble $\L$-series}

Let $G$ be a group.  Let  ${\mathcal L}  =  [(  p_1,  c_1),\ldots, (p_{k},
c_{k}) ]$ be a  list of  pairs consisting  of  a  prime, $p_i,$ and  a
non-negative  integer,  $c_i,$  with  $p_i \not=  p_{i+1}$  and  $c_i$
positive for $i < k.$ For $1 \le i \le k$ and $0 \le j \le c_i$ define
the      list      ${\mathcal      L}_{i,j}      =      [(p_1,c_1),\ldots,
(p_{i-1},c_{i-1}),(p_{i}, j) ].$ Set  ${\mathcal L}_{1,0}(G) =  G.$ For $1
\le i \le k$ and $1 \le j \le c_i$ let
$${\mathcal L}_{i,j}(G) = {\mathcal P}_{j}^{p_i}( {\mathcal L}_{i,0}(G) )$$ 
and for $1 \le i < k$ let
$${\mathcal L}_{i+1,0}(G) = {\mathcal L}_{i,c_i}(G)$$ and ${\mathcal L}(G) = {\mathcal
L}_{k,c_k}(G).$ Note that  ${\mathcal L}_{i,j}(G) \ge {\mathcal L}_{i,j+1}(G)$
holds for $j < c_i.$

\medskip 

The chain of subgroups $$G  = {\mathcal L}_{1,0}(G) \ge {\mathcal  L}_{1,1}(G)
\ge \cdots \ge  {\mathcal  L}_{1,c_1}(G) = {\mathcal L}_{2,0}(G)\ge \cdots \ge
{\mathcal  L}_{k,c_k}(G)  =  {\mathcal L}(G)  $$  is  called the  {\it  ${\mathcal
L}$-series} of $G.$ If ${\mathcal  L}(G) = \langle 1  \rangle$ then $G$ is
an  {\it  ${\mathcal  L}$-group.}  If  $c_k >0$  and $\tilde{\L}(G)  \not=
\langle 1  \rangle$ for $\tilde{\L}  = [ (p_1, c_1  ),  \ldots,  (p_k,
c_k-1)]$ then $G$ is a {\it strict $\L$-group.}

\medskip 
Note that  in the  definition  of strict $\L$-group the exponent-$p_k$
class  of  $\L_{k,0}(G)$   is  determined  but  not  necessarily   the
exponent-$p_i$ class of $\L_{i,0}(G).$ For every  finite soluble group
$G$ there exists a (not necessarily unique) list $\L$ such that $G$ is
a strict $\L$-group.

For  a  given $i$ the series ${\mathcal L}_{i,0}(G) \ge  \cdots  \ge {\mathcal
L}_{i,c_i}(G)$  is  an initial segment  of  the  lower  exponent-$p_i$
central series  of  ${\mathcal  L}_{i,0}(G)$  and ${\mathcal L}_{i,0}(G)/{\mathcal
L}_{i,c_i}(G)$ is a $p_i$-group of exponent-$p_i$ class at most $c_i.$

We use the following notation. 
$$ \L^{+p} = \left\{\,\, \vcenter{ \halign{ $ #$\hfil & \quad $#$\hfil \cr
[ (p_1, c_1),\ldots,(p_{k}, c_{k}+ 1)] 
&\hbox{\ if\ } p  = p_k, \cr
[ (p_1, c_1),\ldots,(p_{k}, c_{k}),(p,1)]
&\hbox{\ if\ } p  \not=p_k. \cr } } \right.$$
Then ${\mathcal  L}(G)/{\L}^{+p}(G)$  is an  elementary  abelian
$p$-group.   Further  
$$\L^{-p} = \left\{\,\, \vcenter{ \halign{ $ #$\hfil & \quad $#$\hfil \cr
[  (p_1,  c_1), \ldots, (p_{k-1}, c_{k-1}) ]
&\hbox{\ if\ } p  = p_k, \cr
[ (p_1,c_1), \ldots, (p_{k}, c_{k}) ]
&\hbox{\ if\ } p  \not=p_k. \cr } } \right.$$

Let  $\pcpG$ be  a power conjugate presentation for a  finite  soluble
group $H$ with  $\A = \{  a_1, \ldots, a_n \}.$ Let $d$ be the minimal
number of generators in $\A$  required  to generate  $H.$ Assume there
exists a $d$-element subset $\X$ of $\A$ such that  $\X$ generates $H$
and for each generator $a \in \A \setminus \X$  there is at least one
relation of  $\R$ having $a$ as the {\em last} generator on  the right
hand  side and occurring with exponent 1.  Choose exactly one of these
relations and call it the {\it definition} of $a.$ The  fact that this
relation is the definition of $a$ is emphasised by  using '=:' instead
of '=' in the relation.  The colon is on the same side of the relation
as the generator defined by this  relation.  The presentation  $\pcpG$
together with chosen definitions for the  generator  in  $\A\setminus
\X$ is called  {\it labelled}.  Let $G$ be a group with generating set
$\{ g_1, \ldots,g_b \}$ and $\tau$ an epimorphism of $G$ onto $H.$ For
$i=  1,  \ldots,  b$  let $w_i$  be  the  normal  word  equivalent  to
$\tau(g_i).$ If  $a\in \X$ is the {\em last} generator in at least one
$w_i$ occurring  with exponent 1 and we have chosen one such $w_i,$ we
write $\tau(g_i) =: w_i$ and call this the {\it definition of a}.  For
each $a\in \X$ there is a  maximal $k$  and a maximal $c$ such that $a
\in \L_{k,c}(H).$ We call $\tau$  a {\it labelled} epimorphism if each
generator $a \in \X$ has a  definition  $\tau(g_i) =: w_i$ and $a$  is
the  only   generator  which   occurs  in  $w_i$  and  does   lie   in
$\L_{k,c}(H).$  If  $\pcpG$ is a labelled power conjugate presentation
for the group $H$ and $\tau$ a labelled epimorphism  from $G$ to  $H,$
then every generator in $\A$ has a definition either as an image under
$\tau$ or as a relation in $\R.$ Further we can read off a preimage in
$G$ for each $a_k\in \A$ under $\tau.$ Thus we can compute a preimage in
$G$ for each $h\in H$ under $\tau.$

The following is  a labelled consistent power conjugate presentation for
$S_4,$ the symmetric group on $4$ letters:
$$\begin{tabular}{llll}
$\{\, a,\, b,\, c,\, d\,\mid\, a^2 =: c,$\\
\hphantom{$\{\,a,\,b,\,c,\,d\,\mid\,\,$}$b^{a}=\hphantom{:} b^2 c,$& $b^3,$\\
\hphantom{$\{\,a,\,b,\,c,\,d\,\mid\,\,$}$c^{a}=\hphantom{:} c,$&$c^b=:d,$&$c^2,$\\
\hphantom{$\{\,a,\,b,\,c,\,d\,\mid\,\,$}$d^{a}=\hphantom{:}cd,$&$d^b=\hphantom{:}cd,
       $&$d^c=d,$&$d^2\,\}.$\\
\end{tabular}$$

The relations $a^2 = c $ and $c^b =  d$ are the definitions of $c$ and
$d,$ respectively.  Note that this  notation implicitly  characterises
the set $\X$ as the  subset of $\A$ whose elements do not occur as the
last element  of a right hand side of  a  definition. In the  example,
$\X$ is the set $\{ a, b \}.$

Consider   the   group   $G$   having  the  finite  presentation  $$\{
x,\,y\,\mid\,x^8,\,y^3,\,(x^{-1}y)^2,\,(yx^3yx)^2  =  x^4\}$$  and let
$\L$ be  the list $[ (2,1), (3,1),  (2,1) ].$ Then the map $\tau$ from
$G$ to $S_4$  defined  by $\tau(x)  =:  a$ and $\tau(y)  =:  b$  is  a
labelled epimorphism.  The elements of $\X$ have definitions as images
of $\tau,$ whereas the other elements in $\A$ have definitions  in the
relations of the power conjugate presentation for $S_4.$

\section{The $\L$-covering group}

Let $\L$ be the  list $[ (p_1,  c_1), \ldots, (p_{k}, c_{k}) ],$ where
$c_i$ is a non-negative integer for $1 \le i \le k,$ and let  $K$ be an
$\L$-group  with generator number $d.$ Let  $F$  be  the free group of
rank $d$ and let $\theta$ be an epimorphism  of $F$ onto  $K.$ Let $p$
be a  prime  and  denote $\L^{-p}(K)$  by $P$ and  let  $F_P$ be   the
preimage in $F$ of $P$ under $\theta.$

Let $K$  be a finite strict ${\mathcal L}$-group with generator number
$d.$ A group $H$ is a \hbox{\it $p$-descendant} of $K$ if $H$ has
generator number $d$ and $H$ is an $\tilde{\L}$-group, where for some
non-negative integer $n$ 
$$\tilde{\L} = \left\{\,\, \vcenter{ \halign{ $ #$\hfil & \quad $#$\hfil \cr
[(p_1, c_1), \ldots, (p_k, c_k + n)]
&\hbox{\ if\ } p  = p_k, \cr
[(p_1,c_1), \ldots, (p_k, c_k), (p,n) ]
&\hbox{\ if\ } p  \not=p_k, \cr } } \right.$$ 
and $H/\tilde{\L}_{k,c_k}(H)$ is isomorphic to $K.$ If $H$
is a strict ${\L}^{+p}$-group, that is $n =  1,$ then  $H$ is  an {\it
immediate $p$-descendant} of $K.$
Note that if $K$ is a $p$-group these definitions are
the same as in O'Brien (1990).

\begin{theorem}
Let $K$ be  a finite strict ${\mathcal L}$-group with generator number $d$
and $p$ a prime. There exists an ${\mathcal L}^{+p}$-group  $\hat{K}$ with
generator number $d$ such that  every  immediate \hbox{$p$-descendant}
of $K$ is isomorphic to a quotient of $\hat{K}.$
\end{theorem}
\begin{proof} 
Let $F$ be the free group of rank $d$ and let $R$ be the kernel of the
epimorphism $\theta$ of $F$ onto $K.$ Let $F_P$ be the  preimage under
$\theta$ in  $F$ of $P=\L^{-p}(K).$ Define $S$ to be $[R, F_P]R^p$ and
define the group $\hat{K}$  to be  $F/S.$ Then $S \le R$ and $\hat{K}$
is a $d$-generator group.  Let  $H$  be an immediate $p$-descendant of
$K$ and let $\nu$ be an epimorphism of $H$ onto $K.$ Then there exists
an  epimorphism   $\hat{\theta}$   from   $F$   to   $H$   such   that
$\hat{\theta}\nu  =  \theta.$   Since  $R\hat{\theta}\le  \ker\nu$  it
follows  that $R\hat{\theta}$ is an elementary abelian $p$-subgroup of
$H,$ which is central in \F{P}$\hat{\theta}.$ Hence $S\hat{\theta} = [
R\hat{\theta}, F_P\hat{\theta} ] (R\hat{\theta})^p$ is the identity in
$H$ and thus $H$ is a homomorphic image of $F/S.$
\end{proof}

If  $p$ is not $p_k$ then $P$  is  the  identity  and  $F_P$  is  $R.$
Therefore $S  = [R,R]R^p$ and $R/S$ is  the  relation module of  $F/R$
(see Gruenberg, 1976).

The group  $\hat{K} =  F/S$ with  $S =  [ R,  F_P  ]  R^p$ is the {\it
$\L$-covering  group  of  $K$  with  respect  to  the  prime  $p.$}  A
consistent  power conjugate  presentation $\pcphG$ for $\hat{K}$ is an
{\it $\L$-covering presentation} of $\pcpG.$ It should always be clear
from the context which prime $p$  is chosen.   Therefore $\hat{K}$  is
called  the $\L$-covering  group of $K$ without  reference to $p.$  If
${\mathcal L}$ is the list $[ ( p, c )]$ then an ${\mathcal L}$-group $K$ is a
$p$-group  and the  $\L$-covering group  $\hat{K}$ is the $p$-covering
group $K^*$ of  $K.$ O'Brien (1990) shows  that $K^*$ is isomorphic to
$F/[R, F] R^p.$

The $\L$-covering group $F/S$ is the largest extension  of $F/R$ by an
elementary  abelian  $p$-group such  that the extension  has the  same
generator number as $F/R$ and $F_P/R$ acts trivially on the elementary
abelian $p$-group. Note that the  generator number  of $F_P/S$ may  be
larger than the generator number  of  $F_P/R.$ The $p$-covering  group
$(F_P/R)^*$  of  $F_P/R$  is the largest  extension  of $F_P/R$  by an
elementary  abelian  $p$-group which has the same generator  number as
$F_P/R$  and  $F_P/R$   acts  trivially  on  the  elementary   abelian
$p$-group.

The following theorem asserts  that the isomorphism type of  $F/S$  is
independent of the choice of the homomorphism $\theta$ from $F$ to $K$
and thereby  independent of its kernel $R.$  It is valid  also for the
case that $\L^{-p}(F/R)$ is trivial.

\begin{theorem}
Let $R_1$ and $R_2$ be normal subgroups of $F$ such
that $F/R_1$ and $F/R_2$  are $\L$-groups.\
Furthermore, let $U_1$ and $U_2$ be subgroups of $F$ such that for
$i \in \{ 1,2 \}$\\
{\rm \hphantom{iiiiii}1)} $R_i \le U_i,$\\
{\rm \hphantom{iiiiii}2)} $U_i/R_i \le \L^{-p}(F_P/R_i),$ \\
{\rm \hphantom{iiiiii}3)}  $U_i/R_i$ is characteristic in $F/R_i,$\\
\noindent
and  there  exists an isomorphism  $\varphi$ of  $F/R_1$ onto $F/R_2,$
which  maps  $U_1/R_1$  onto $U_2/R_2.$  Define  $S_i$  to  be  $[R_i,
U_i]R_i^p$ for $i \in  \{1,2\}.$ Then $F/S_1$ is isomorphic to $F/S_2$
by an isomorphism which  takes $R_1/S_1$ to $R_2/S_2$ and $U_1/S_1$ to
$U_2/S_2.$
\end{theorem}

\begin{proof}
Let  $\{  a_1, \ldots, a_d \}$  be a free generating set for  $F.$ Let
$\nu$ be the canonical epimorphism of $F/S_1$  onto $F/R_1.$ Then $\nu
\varphi$ is  an epimorphism from $F/S_1$ onto $F/R_2.$ Let $b_i \in F$
such that  $(b_iS_1) \nu\varphi = a_iR_2.$ Define a homomorphism $\rho
\!  :\!  F \rightarrow F/S_1$ mapping $a_i$ to $b_iS_1.$ The map $\rho
 \nu \varphi$ is an epimorphism from $F$ onto $F/R_2.$ Since $\rho \nu
\varphi$  agrees on each $a_i$ with the natural projection of $F$ onto
$F/R_2$  and  since the $a_i$ generate $F,$ $\rho \nu  \varphi$ is the
natural projection.  Thus $R_2\rho \nu \varphi = 1,$ and $R_2 \rho \le
\ker  \nu\varphi = R_1/S_1$  and  $U_2 \rho \nu \varphi = U_2/R_2,$ so
$U_2\rho  \le (U_2/R_2) \varphi^{-1} \nu^{-1}  =  (U_1/R_1)\nu^{-1}  =
U_1/S_1.$  Therefore  $S_2\rho  =  ([R_2,  U_2]  {R_2}^p)  \rho$  is a
subgroup of $[ R_1/S_1, U_1/S_1 ] (R_1/S_1)^p.$ Since the elements  of
$R_1/S_1$ are of order  $p$ and commute  with $U_1/S_1,$ we  have that
$S_2 \le \ker\rho.$ Hence $F/S_1$  is isomorphic to  a factor group of
$F/S_2.$ Similarly $F/S_2$ is isomorphic to a factor group of $F/S_1,$
and therefore $F/S_1 \cong F/S_2.$
\end{proof}

If  the  subgroup $U_i$ is  chosen to  be  $\L^{-p}(F_p/R_i)$ then the
theorem asserts that given a $d$-generator  $\L$-group  $G$ with  $P =
\L^{-p}(G),$ the choice of the epimorphism $\theta :  F \rightarrow G$
and thus the   choice of  $R = \ker   \theta$ does not  influence  the
isomorphism type of $F/S.$ We can also choose $U_i = R_i$ and thus the
choice of   the  epimorphism  $\theta$ also   has  no  impact  on  the
isomorphism type of $F/[R,R]R^p.$

\section{An $\L$-covering algorithm}

The  task of the  $\L$-covering  algorithm is  to determine a labelled
consistent power conjugate presentation for the $\L$-covering group of
a soluble $\L$-group $K.$

\begin{itemize}
\item The  input   of   the $\L$-covering  algorithm   is a
labelled consistent  power conjugate presentation for  a soluble
$\L$-group $K$ and a prime $p.$
\item The output is a consistent power conjugate
presentation for the soluble group $\hat{K},$ the  $\L$-covering group
of $K$ with respect to the prime $p.$ 
\end{itemize}

\noindent The $\L$-covering algorithm presented here first computes  a
finite presentation for  $\hat{K}.$ It is shown that one can define  a
normal  form  for  the elements  in  $\hat{K}$  and  that  the  finite
presentation can  be  used  like a  power  conjugate  presentation  to
compute normal forms of  elements of  $\hat{K}.$  Applying  a  theorem
which  is   a  generalisation   of  Theorem  \ref{consis}  allows  the
determination of  a module  presentation for the kernel of the natural
epimorphism  of $\hat{K}$  onto  $K.$  A  vector space basis  for this
\Fp$K$-module is computed by an algorithm,  called vector enumeration,
and  this in turn  enables the determination  of a labelled consistent
power conjugate presentation for $\hat{K}.$

The individual steps of the algorithm are illustrated by reference to
the example of the symmetric group on 4 letters, $S_4.$ A consistent
power conjugate presentation for this group was given in Section 2.
For this example let $\L$ be the list $[ (2,1), (3,1),  (2,1) ]$ and
let $p$ be the prime $2.$

\subsection{A  finite presentation for the $\L$-covering group}

A finite presentation $\pcptG$ for $\hat{K} =
F/S$ is obtained in the following way.
Let $\pcpG$ be the labelled consistent  power conjugate presentation
for the $\L$-group $K,$ where  $\A$ is the set $\{ a_1,\ldots, a_n \}$ and
$$\R = \{ a_i^{p_i}= v_{ii},\,
	a_k^{a_j}= v_{jk} \mid
1\le i\le n,\,1\le j<k\le n\}.$$
Let $P$ be $\L^{-p}(K).$ 
Assume that  $\pcpG$ is a  power conjugate presentation  for  $K$ with
respect  to a composition   series which  refines  the soluble  ${\mathcal
L}$-series.  Therefore  there exists  an $r$ such that
$\{ a_1 P, \ldots,  a_rP \}$ generates $K/P$  and $\{ a_{r+1}, \ldots,
a_n  \}$  generates $P.$
Let $s$ be the number  of relations in $\R$  which are not  definitions.
Then $s   =
{(n-1)n/{2}} + d.$  Introduce new generators $\{y_1, \ldots, y_s\}$ and
define $\tilde{\A} = \{a_1, \ldots, a_n \} \cup \{y_1, \ldots, y_s\}.$ We
obtain $\tilde{\R}$ in the following way:

\begin{enumerate}
\renewcommand{\theenumi}{\arabic{enumi}}
\item  initialise $\tilde{\R}$ to contain all relations of $\R$
which  are definitions;    
\item  modify  each non-defining relation $a_i^{p_i}  =  v_{ii}$
or $a_k^{a_j} = v_{jk}$ of $\R$ to read
$a_i^{p_i}  =  v_{ii}y_t$ or $a_k^{a_j} = v_{jk}y_t$ 
for some $t \in \{  1,  \ldots, s\},$ where different
non-defining relations  are modified by  different $y_t,$
and add the modified relation to $\tilde{\R};$ 
\item  add to $\tilde{R}$ all relations of  the form
$[y_{i}^{},y_{j}^g]=1$ for all normal $ g = w(a_1, \ldots,
a_r)$  and  $ y_{i}^p  = 1$ for $ 1 \le i,j \le s;$ 
\item  add to  $\tilde{\R}$ all  relations $y_{i}^{a_j}  = y_i^{}$
for  $j >  r$  and $1 \le i \le s.$ 
\end{enumerate}

We apply this to the example of $S_4.$ The subgroup $\L_{3,0}(S_4)$ is
a  $2$-group isomorphic to  the Klein 4-group and is  generated by $c$
and $d.$ Further, $S_4$ is generated by $a$ and $b.$ The definition of
$c$ is  the relation with left-hand side $a^2$  and the  definition of
$d$  is the  relation with  left-hand  side  $c^{b}.$  We  obtain  the
following presentation for $\hat{S_4}:$

$$\begin{tabular}{rllllr}
\multicolumn{6}{l}%
{$\tilde{\A}=\{a,\,b,\,c,\,d,\,y_1,\,y_2,\,y_3,\,y_4,\,y_5,\,y_6,\,y_7,%
\,y_8\}$ and}\\
$\tilde{\R} = \{ $ & $a^2 = c,$\\
                   & $b^{a}= b^2 c y_1,$ & $b^3=y_2,$\\
                   & $c^{a}= c y_3,$ & $c^{b}=d,$ & $c^2=y_4,$\\
                   & $d^a=cdy_5,$&$d^b=cdy_6,$&$d^c=d y_7,$&$d^2=y_8,$ \\
                   & $y_i^{c}= y_i,$&  
                     \multicolumn{3}{l}{ for $1 \le i \le 8,$ }\\
                   & $y_i^{d}= y_i,$&  
                     \multicolumn{3}{l}{ for $1\le i \le 8,$ }\\
                   & $[ y_i, y_j ^g ] = 1$& 
                     \multicolumn{3}{l}{ for $1 \le i,j \le 8$  and $g
                     \in \{ a, a b, b, b^2, ab^2 \}$ } & $\}.$\\
\end{tabular}$$

Consider the following diagram. 
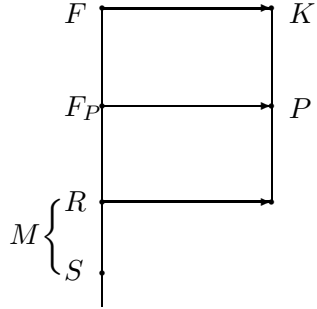
\begin{center}
\setlength{\unitlength}{0.0125in}%
\begin{picture}(114,142)(15,681)
\thinlines
\put( 52,739){\circle*{2}}
\put( 52,779){\circle*{2}}
\put( 52,820){\circle*{2}}
\put(123,739){\circle*{2}}
\put(123,779){\circle*{2}}
\put(123,820){\circle*{2}}
\put( 52,709){\circle*{2}}
\put( 52,779){\vector( 1, 0){ 71}}
\put(123,779){\line(-1, 0){ 71}}
\put( 52,820){\vector( 1, 0){ 71}}
\put(123,820){\line(-1, 0){ 71}}
\put(123,820){\line( 0,-1){ 81}}
\put(123,739){\line( 0, 1){ 81}}
\put( 52,739){\vector( 1, 0){ 71}}
\put(123,739){\line(-1, 0){ 71}}
\put( 52,820){\line( 0,-1){125}}
\put( 35,814){\makebox(0,0)[lb]{\raisebox{0pt}[0pt][0pt]{\small $F$}}}
\put( 35,775){\makebox(0,0)[lb]{\raisebox{0pt}[0pt][0pt]{\small $F_P$}}}
\put(129,814){\makebox(0,0)[lb]{\raisebox{0pt}[0pt][0pt]{\small $K$}}}
\put(129,774){\makebox(0,0)[lb]{\raisebox{0pt}[0pt][0pt]{\small $P$}}}
\put( 35,735){\makebox(0,0)[lb]{\raisebox{0pt}[0pt][0pt]{\small $R$}}}
\put( 35,706){\makebox(0,0)[lb]{\raisebox{0pt}[0pt][0pt]{\small $S$}}}
\put( 12,721){\makebox(0,0)[lb]{\raisebox{0pt}[0pt][0pt]{\small $M$}}}
\put( 25,721){\makebox(0,0)[lb]{\raisebox{0pt}[0pt][0pt]{$\small \left\{
\vphantom{\line(0,1){17}}\right.$}}}
\end{picture}
\end{center}

The subgroup  $R/S,$  denoted  by $M,$ is  the  kernel of  the natural
epimorphism of $\hat{K}$ to  $K$ and can be characterised as  follows.
It is the maximal  \Fp$K$-module by which  $K$ can be extended so that
$P$  acts trivially on  $M$ and the extension has  the same  generator
number as $K.$ Thus $M$  is an \Fp$(K/P)$-module.  Let $Y$ be the free
\Fp$(K/P)$-module  on $\{ y_1,  \ldots, y_s \}.$ The  module $M$ is  a
homomorphic image of $Y.$ The kernel of the homomorphism from $Y$ onto
$M$  can be  computed effectively.  In order to see this, we study the
finite presentation for the group $\hat{K}$ in more detail.

\subsection{Collecting in $\hat{K}$}

One  can  collect in the  group  $\hat{K}$  relative to $\pcptG.$  The
definition of a normal word  can be generalised  for this presentation
in the following way.  A word in $\tilde{\A}$ is {\it normal\/} if  it
is  of the form  $w(a_1, \ldots, a_n  ) \cdot \Pi_{i=1}^s  y_i^{f_i},$
where $w( a_1, \ldots, a_n)$ is a normal  word in $\{ a_1, \ldots, a_n
\}$  and $f_i$  is an  element  of  \Fp$(K/P).$  The  following steps,
referred to as  ``collection  in $\hat{K}$'',  can be applied to every
word in $\tilde{\A}.$ For $f,\, f'  \in $ \Fp$(K/P)$ and $1 \le k,\, l
\le s$ and $ 1 \le i,\, j \le n$

\begin{enumerate}
\renewcommand{\theenumi}{\arabic{enumi}}
\item replace  $y_l^f y_k^{f'}$ by $y_k^{f'} y_l^{f}$ for $k < l;$

\item replace $y_k^f y_k^{f'}$ by $y_k^{f+f'};$

\item \begin{enumerate}
\renewcommand{\theenumi}{\alph{enumi}}
\item replace  $y_k^f a_i^{q}$ by $a_i^{} y_k^{(f a_i)}
a_i^{q-1}$ if $q > 1;$

\item replace  $y_k^f a_i^{}$ by $a_i^{} y_k^{(f a_i)}; $
\end{enumerate}

\item replace  $a_i^{p_i}$ by $v,$ where $a_i^{p_i} = v$ 
is a relation in $\tilde{\R};$

\item replace $a_j a_i$ by $a_i v,$ where $a_j^{a_i}=v$ 
is a relation in $\tilde{\R}$ for $i < j.$
\end{enumerate}

\noindent In each step a word is replaced by another word representing
the same element of $\hat{K}.$ After applying a finite number of these
steps to any word it is replaced by a normal word.  This can be proved
in a way similar  to  proving  that a  collection  process  computes a
normal word after applying finitely many collection steps.  Rules  1),
2) and 3) use the fact that $M$ is an \Fp$(K/P)$-module.  Note that 4)
and 5) resemble collection steps in a collection algorithm, where  the
power conjugate presentation is used to determine the replacement.

For example the word $b ( ba)$ in collects in $\hat{S_4}$ to 
$ a b d  y_1^{b^2+1} y_2^b y_4 y_6  y_7.$

The following lemma states that two equivalent normal words 
can differ only by a module word in $Y.$

\begin{lemma}
\label{normalform}
Let $w$ be an arbitrary word in $\{ a_1, \ldots,
a_n \} \cup \{ y_1, \ldots, y_s \}.$ Then there exists a unique normal
word $v$ in $\{ a_1, \ldots, a_n \}$ such that any normal word
in $\hat{K}$ equivalent to $w$ has the form $v\cdot \Pi_{i=1}^s y_i^{f_i}.$
\end{lemma}

\begin{proof}
The existence of the unique normal word $v$ follows from the fact that
$\pcpG$ is a consistent power conjugate presentation for $K.$ 
\end{proof}

A consequence of this lemma is that the  map $\phi : \hat{K} \rightarrow K$
which  maps a  word in  $\hat{K}$ to its  unique  normal  word in $\{  a_1,
\ldots, a_n \}$ is an epimorphism.

The following theorem allows  us to describe the kernel of the homomorphism
from $Y$ onto  $M=R/S$ in a manner suitable  for computation.  It considers
certain non-normal words in $\hat{K}.$

\begin{theorem}
\label{consisrels}
Let $Y$ be the  free  \Fp$(K/P)$-module on $\{ y_{1},
\ldots, y_{s}\}$ and $\pcptG$  the presentation
for  the   extension $\hat{K}$ of $K$  as defined above.
Let $W$ be the following set of consistency test words in $\{a_1,
\ldots, a_n\}:$ 
$$\vbox{\openup 3pt
\halign{\hfil$#$\hfil &\quad{\it #}\hfil& \hskip .5em$#$\hfil\cr
(\bigl(\vphantom{a_k^{p-1}}(a_k^{}\, a_j^{})\, a_i^{}\bigr) \, 
\bigl( \vphantom{a_k^{p-1}}a_k^{}\,(a_j^{}\, a_i^{})\bigr)^{-1}) & for
& 1 \le i < j < k \le n,\cr
(\bigl(\vphantom{a_k^{p-1}}(a_k^p)\, a_j^{} \bigr)\, \bigl( a_k^{p-1}\,
(a_k^{}\, a_j^{})\bigr)^{-1}) & for 
&1 \le j < k \le n,\cr
(\bigl(\vphantom{a_k^{p-1}}(a_j^{}\,a_i^{})\,a_i^{p-1} \bigr)\,
\bigl(\vphantom{a_k^{p-1}} a_j^{}\,(a_i^p)\bigr)^{-1}) & for 
&1 \le i < j \le n,\cr 
(\bigl(\vphantom{a_k^{p-1}} (a_i^p)\, a_i^{}\bigr)\, \bigl(
\vphantom{a_k^{p-1}}a_i^{}\, 
(a_i^p)\bigr)^{-1}) & for &  1 \le i \le n\cr}}$$

\noindent and  let $T$ be  the set of elements obtained  by collecting
the words in $W$ with respect to $\pcptG.$ Then $T$  consists of words
in $Y$ and $M$ is isomorphic to $Y/(T $\Fp$(K/P)).$
\end{theorem}

\begin{proof}
The elements  of $W$ represent  the identity element in $\hat{K}.$  By
Lemma \ref{normalform} the  elements of $T$  are words in  $Y.$ Denote
$T$\Fp$(K/P)$ by $\langle T\rangle.$ Let $\mu :  Y  \rightarrow M$  be
the  epimorphism mapping  $y_i$  in  $Y$ to $y_i$  in  $M.$  Since the
elements of  $W$ are the identity  in  $\hat{K}$  it follows that  the
elements of  $T,$ when  viewed  as elements of $Y,$  are mapped to the
identity element of $M,$ hence $\langle T \rangle \subseteq \ker \mu.$
Therefore $Y/\langle T\rangle$ has a factor module isomorphic to $M.$

We  now  define a  consistent  power  conjugate presentation  for  $K$
extended  by  $Y/\langle  T\rangle$  such  that  the  extension  is  a
$d$-generator group. Since  $M$ is the largest  \Fp$(K/P)$-module with
these properties it follows that $Y/\langle T\rangle$ is isomorphic to
$M.$

Let $\{ b_1, \ldots, b_m \}$ be a vector space basis for  $Y/\langle T
\rangle.$ Then each  element $y_i \langle T  \rangle$ can be expressed
uniquely in  the basis elements. Therefore we obtain a power conjugate
presentation $\pcphG$ for an  extension of $K$ by $Y/\langle T\rangle$
in the  following way from  the presentation $\pcptG$ where $\hat{\A}$
is $\A \cup \{ b_1, \ldots, b_m \}:$
\begin{enumerate}
\renewcommand{\theenumi}{\arabic{enumi}}
\item  replace every occurrence of an element $y_i$  on the right hand
side of a relation in $\tilde{\R}$ with left hand side a word in $\A,$
by the corresponding word for $y_i \langle T\rangle $ in the basis and
add this modified relation to $\hat{\R};$
\item add to $\hat{\R}$ the relations $b_i^{a_j} = w_i(b_1,
\ldots, b_m)$ for all $1 \le i \le m$ and $1 \le j \le r,$ where $w_i(b_1,
\ldots, b_m)$ is determined by the action of $a_j$ on $b_i;$
\item add to $\hat{\R}$ the relations $b_i^{a_j} = b_i^{}$ for
$1\le i\le m$ and $r+1 \le j \le n;$
\item add to $\hat{\R}$ the relations $b_j^{b_i} = b_j^{}$ for $1
\le i < j \le m$ and the relations $b_i^{p} = 1$ for $1\le i \le m.$
\end{enumerate}

\noindent
We now show  that $\pcphG$ is consistent.   Let $\tilde{K}$ denote the
group  defined  by  $\pcphG.$ Since the  elements  of  $W$  collect to
elements   of  $\langle  T\rangle$  they  are   the  trivial  word  in
$\tilde{K}.$

The   consistency  of   $\pcphG$  is   proved   by   applying  Theorem
\ref{consis}.  We  only need  to consider consistency relations  which
involve at  least one element  of the basis.  Any consistency relation
which  involves only basis elements holds, since the  basis is a basis
for a vector space over \Fp.

Consider the word $b_k  a_j a_i.$  Applying a collection step to $(b_k
a_j) a_i$ with respect  to  $\pcphG$  yields $a_j b_k^{a_j} a_i$ which
collects  to $a_j a_i  (b_k^{a_j})^{a_i}$ and  finally to $a_i  v_{ij}
(b_k^{a_ja_i}).$ On the other hand $b_k (  a_j a_i) $ collects to $b_k
a_i  v_{ij}$  which  in  turn  collects  to  $a_i  v_{ij}  b_k  ^{(a_i
v_{ij})}.$ Since $(b_k^{a_ja_i})$ and $b_k  ^{(a_i  v_{ij})}$ are  the
same module element they have the same normal form in the basis.

Consider the word $b_k b_j  a_i.$ Applying a collection  step to $(b_k
b_j) a_i$ with respect to $\pcphG$ yields $b_j b_k a_i$ which collects
to  $a_i b_j^{a_i}  b_k^{a_i}.$ On the  other  hand  $b_k ( b_j  a_i)$
collects  to  $b_k a_i b_j^{a_i}$  which  collects to  $a_i  b_k^{a_i}
b_j^{a_i}.$  Therefore  the  first  consistency  relation  in  Theorem
\ref{consis} holds.  Similarly one can prove that  the other relations
also  hold  and therefore  $\pcphG$  is a  consistent power  conjugate
presentation.  \end{proof}

In  collecting  the  words in  $W$  with  respect to the  presentation
$\pcptG$  we obtain  a set  $T$  which  generates the  kernel  of  the
epimorphism of the free \Fp$(K/P)$-module $Y$ onto  the module $M.$ In
the example $T$ is the set:
$$\begin{tabular}{llll}
$\{ y_2,$& $y_1^{(1+b+b^2)}y_3 y_5 y_6 y_7,$ &
     $y_3^{(a+1)},$ & $y_3 y_6 y_7^{(a+1)},$\\
    $y_3^b y_7,$ &  $y_3 y_6 y_7^{(1+b)},$ &
    $y_2^{(a+1)},$ & $y_2 y_3 y_4^{(a+1)} y_6,$ \\
\multicolumn{4}{l}{$y_1^{(1+a+b^2)} y_2 y_3^{(1+b^2)}y_5^{(1+b+b^2)}
y_6^{(1+b^2)}y_7,$} \\
$    y_5^{(1+b)} y_6 y_7,$&$ y_5^{(b+b^2)} y_6^b y_7^b,$&
$    y_4y_5y_6,$&$      y_6^{(1+a)}, $\\
\multicolumn{3}{l}{$y_2 y_3
y_4^{(1+b^2)}y_5^{(a+b)}y_6^{(1+b)}y_7^b,$} &
$y_6^{(1+b)} \}.$\\
\end{tabular}$$

In the proof of the  previous theorem it  was  assumed  that we have
\begin{enumerate}
\renewcommand{\theenumi}{\arabic{enumi}}
\item vector space basis $\{ b_1, \ldots, b_m \}$ for the
\Fp$(K/P)$-module $Y/\langle T\rangle;$ 
\item an expression in the basis for $b_i^{a_j}$ for $1 \le i
\le m$ and $1\le j \le r;$
\item an expression in the basis for $y_i\langle T\rangle$
for $1\le i\le s.$
\end{enumerate}

\noindent Where such  information  is  available  the proof  yields a
constructive   method   to   obtain  a   consistent   power  conjugate
presentation  $\pcphG$  for  $\hat{K}.$ We  now  describe an algorithm
which may be used to obtain this information.

\subsection{Computing a vector space basis for a module}

The technique of vector enumeration is used to compute a basis
for   the  \Fp$(K/P)$-module  $M$
needed to obtain a power conjugate presentation for $F/S.$
A vector enumeration algorithm is   described in
Linton   (1991) and Linton (1993).  Its use in this context has been 
suggested by Leedham-Green  (private  communication,  1991).

It is used with the following input:
\begin{enumerate}
\renewcommand{\theenumi}{\arabic{enumi}}
\item a consistent power conjugate presentation for  $K;$
\item the  set of  free generators for $Y;$ 
\item a set $T.$ 
\end{enumerate}
\noindent The output is: 
\begin{enumerate}
\renewcommand{\theenumi}{\arabic{enumi}}
\item an \Fp-basis $\{ b_1, \ldots, b_m \}$ for $M;$ 
\item the matrix action of each generator of $K/P$ in the
power conjugate presentation of $K/P$ on $M$ with respect
to the computed basis;
\item an expression in the computed basis for $y_i \langle
T\rangle$ for $1 \le i \le s.$
\end{enumerate}

This    output is used   to   obtain   a consistent  power   conjugate
presentation for the extension   $\hat{K}$  of $K$  by $M$   using the
method described in the proof of Theorem \ref{consisrels}.

We illustrate the technique  by reference to  our example.  The vector
enumerator with input $\pcpG,$ the  set  $\{ y_1, y_2, y_3,  y_4, y_5,
y_6, y_7 \}$ and $T$ as  above computes a module  basis for the module
$M.$ The basis has five elements $\{ e, f, g, h, i \}$ defined by $e =
y_1,\, f  =  y_2, \,  g = y_3,\, h  = e^a$ and $i  = e^b.$ Further the
vector enumerator gives the action of $a$ and $b$ on the module basis,
while the elements   $c$  and  $d$  act  trivially.   The  information
returned by the   vector enumerator can   be  used   to construct  the
following  consistent power     conjugate   presentation    for    the
$\L$-covering group $\hat{S_4}:$

{\small
$$\renewcommand{\tabcolsep}{.8pt}
\begin{tabular}{llllllllll}
\multicolumn{10}{l}{$\{ a,\,b,\,c,\,d,\,e,\,f,\,g,\,h,\,i,\,j \mid $}\\
$a^2=c,$\\
$b^a=b^2 c e,$&$ b^3,$\\
$c^a=c,$&$ c^b=d,$&$ c^2=f,$\\
$d^a=c d g,$&$ d^b=c d h,$&$ d^c=d g h,$&$d^2=i,$\\
$e^a=j,$&$           e^b=j,$&$     e^c=e,$&$e^d=e,$&$ e^2,$\\
$f^a=f,$&$         f^b=i,$&$     f^c=f,$&$f^d=f,$&$ f^e=f,$&$f^2,$\\
$g^a=fh,$&$       g^b=hi,$&$    g^c=g,$&$g^d=g,$&$ g^e=g,$&$g^f=g,$&$g^2,$\\
$h^a=fg,$&$  h^b=gi,$&$ h^c=h,$&$h^d=h,$&$ h^e=h,$&$h^f=h,$&$h^g=h,$&$h^2,$\\
$i^a=fg hi,$&$ i^b=fghi,$&$ i^c=i,$&$i^d=i,$&$
i^e=i,$&$i^f=i,$&$i^g=i,$&$i^h=i,$&$i^2,$\\ 
$j^a=e,$&
$j^b=efgij,$&$j^c=j,$&$j^d=j,$&$j^e=j,$&$j^f=j,$&$j^g=j,$&$j^h=j,$&$j^i=j,$
&$j^2$\}.\\ 
\end{tabular}$$
}

From the presentation we can read off that $\hat{S_4}$ has order  $2^9
3  = 1536.$ The  group  $\L^{-2}(\hat{S_4})$  has order  $2^8$  and is
generated by $\{ c, d, e, f, g,  h, i, j \}.$ It is the direct product
of the normal subgroups $\langle c, d\rangle$  and  $\langle e, j, fgi
\rangle$ of $\hat{S_4}.$ In general, the complete  preimage  $\hat{P}$
of      $P=\L^{-p}(K)$      is      a      normal     subgroup      of
$\vphantom{P_1}\L^{-p}(\hat{K}).$  It   contains  a  normal  subgroup,
namely $ \hat{P}  \cap M.$  Note that $r$  was  defined  such that the
subset $\{ a_{r+1}, \ldots, a_n\}$ of $\A$ generates $P.$ A generating
set for $\hat{P} \cap M$ is the union of the set $\{ v_{ii} \mid a^p =
v_{ii} \hbox{\ is\ a\ relation\ in\ } \hat{\R}\hbox{\ for\ } i > r \}$
and the  set  $\{  v_{ij} \mid  a_i^{a_j} = a_iv_{ij}  \hbox{\ is\  a\
relation\ in\  } \hat{\R} \hbox{\ for\  } i,j  >r \}.$ It can thus  be
obtained from the power conjugate presentation for $\hat{K}.$

The group ring \Fp$(K/P)$ in the example is  isomorphic to \F{2}$S_3.$
We can  investigate  the module structure of  $M$  as an  $S_3$-module
further.  The submodule $\hat{P} \cap M$ is the direct sum of $\langle
fi, ghi \rangle$  and  $\langle gh  \rangle.$  Its module complement $
\langle e, j, fgi \rangle$ is a direct sum of a one dimensional and  a
two dimensional module.  It has the decomposition $\langle fgi \rangle
\oplus \langle ej, efgi \rangle.$

\subsection{Obtaining a labelled presentation}

In some cases additional work is necessary to transform the consistent
power  conjugate presentation  $\pcphG$ of  $\hat{K}$  into a labelled
presentation.  It is possible that a basis vector $b_i$ does not occur
as the last element of the right hand side of a relation in $\hat{\R}$
and thus no relation can be chosen as the definition of $b_i.$ In this
case we  proceed as  follows. For each basis  vector  $b_i$  choose  a
relation which contains $b_i$ in its right hand side as the definition
of $b_i,$ ensuring that this relation is  not chosen as the definition
of any other basis vector. Assume that for the element $b_i$ the right
hand  side  of its defining  relation  has the form $$w_1(a_1, \ldots,
a_n)  \cdot w_2(b_1,  \ldots, b_{i-1})  \cdot  w_3(b_i,\ldots, b_m).$$
Define  the  element  $\tilde{b}_i$  to  be  $w_3(b_i, \ldots,  b_m).$
Obviously  $\{\tilde{b}_1,\ldots, \tilde{b}_m \}$  is  again a  vector
space basis for $M$ and  the action of the generator $a_j$ of $K/P$ on
this basis can  be computed as the action of $a_j$ on $w_3(b_i,\ldots,
b_m)$ and then  expressing the  result  in  the new basis.  A labelled
consistent power conjugate presentation for $\hat{K}$  is  obtained by
performing a base change.

The  power  conjugate presentation for   $\hat{S_4}$  given  above  is
already a labelled power conjugate presentation, where $c,\, d,\, e,\,
f,\, g,\, h,\, i$ and $j$ are defined by the relations  with left hand
sides $a^2,\, c^b,\,   b^a,\, c^2,\, d^a,\,   d^b,\,  d^2$ and  $e^a,$
respectively.  Every  extension of  $S_4$  by an    elementary abelian
$2$-group $M$ such that the  Klein 4-group acts  trivially on $M$  and
the extension has generator number $2$ is isomorphic  to a quotient of
the group defined by this presentation.

\section{A soluble quotient algorithm}

The  soluble  quotient  algorithm  presented  here  computes  a  power
conjugate presentation for a quotient $G/{\mathcal  L}(G)$ of  a  finitely
presented  group $G,$ where  the  presentation  exhibits a composition
series of the  quotient  group  which  is a refinement  of the soluble
${\mathcal L}$-series.  It takes as input:

\begin{enumerate}
\renewcommand{\theenumi}{\arabic{enumi}}
\item a finite  presentation  $\{ g_1, \ldots, g_b \mid r_1( g_1,
\ldots,  g_b), \ldots, r_m(  g_1,\ldots, g_b) \}$ for $G;$ 
\item a  list ${\mathcal L} =  [(p_1, c_1), \ldots, (p_k, c_k)],$
where each $p_i$ is a  prime,  $p_i \not=  p_{i+1},$ and each 
$c_i$ is a positive integer.
\end{enumerate}
\noindent The output is: 
\begin{enumerate}
\renewcommand{\theenumi}{\arabic{enumi}}
\item a labelled power conjugate presentation for $G/{\mathcal L}(G)$ 
exhibiting a composition series refining the soluble 
${\mathcal L}$-series of this quotient;
\item a labelled epimorphism $\tau : G$\epim$G/{\mathcal L}(G).$
\end{enumerate}

The algorithm  proceeds by computing power conjugate presentations for
the  quotients  $G/{\mathcal  L}_{i,j}(G)$   in  turn.   Without  loss  of
generality  assume  that a power conjugate presentation  for  $G/{\mathcal
L}_{i,j}(G)$ has been computed for $j < c_i.$ The basic  step computes
a power  conjugate presentation for $G/{\mathcal L}_{i,j+1}(G).$ The group
${\mathcal L}_{i,j}(G)/{\mathcal L}_{i,j+1}(G)$ is a $p_{i}$-group.  The basic
step takes as input:

\begin{enumerate}
\renewcommand{\theenumi}{\arabic{enumi}}
\item the finite presentation  for $G;$ 
\item  a  labelled consistent power conjugate presentation for the
finite soluble quotient $K \cong G/{\mathcal L}_{i,j}(G)$ of $G$ 
	  with $j < c_i$ which refines the ${\mathcal L}$-series of $K;$
\item a labelled epimorphism $\theta : G$\epim$K.$
\end{enumerate}

\noindent
The output is:
\begin{enumerate}
\renewcommand{\theenumi}{\arabic{enumi}}
\item a labelled consistent power conjugate presentation for the
finite soluble   
group $H \cong G/{\mathcal L}_{i,j+1}(G),$ exhibiting a 
composition series refining the ${\mathcal L}$-series of $H;$ 
\item an  epimorphism $\phi : H$\epim$K;$
\item a labelled epimorphism $\tau : G$\epim$H$ with $\tau \phi = \theta.$
\end{enumerate}

If during  the basic step it  is discovered that ${\mathcal L}_{i,j}(G)  =
{\mathcal  L}_{i,j+1}(G),$ then  ${\mathcal  L}_{i+1,0}(G)$  is  set to ${\mathcal
L}_{i,j}(G).$

The basic step  is illustrated  by  the  following diagram, where  the
input is  described  on the left and the output  is described  on  the
right.   Put  $p =  p_{i},$ let $P$  denote  ${\mathcal L}_{i,0}(K),$  and
$\hat{P}$  denote ${\mathcal L}_{i,0}(H).$ If $j = 0$ then $P$ is trivial.
The  elementary abelian $p$-group $\ker \phi$ is  denoted by $N.$  The
group $\hat{P}$ acts trivially on  $N;$ thus  $\hat{P}$  is  a central
extension of $P$ by $N,$ and $\hat{P}$ is a $p$-group of  exponent-$p$
class at most one larger than the exponent-$p$ class of $P.$

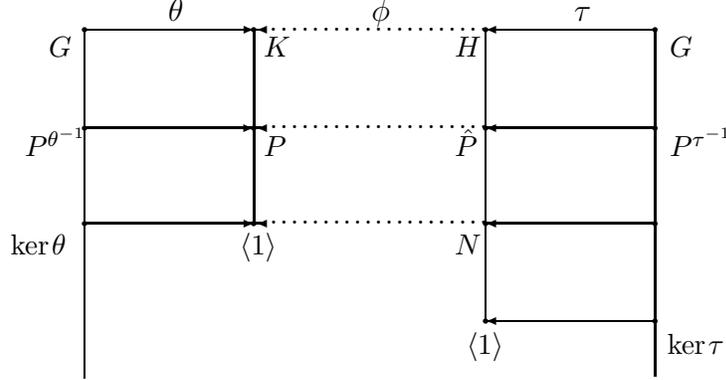
\begin{center}
\setlength{\unitlength}{0.0125in}%
\begin{picture}(241,147)(20,675)
\thinlines
\put( 21,740){\circle*{2}}
\put( 21,780){\circle*{2}}
\put( 21,821){\circle*{2}}
\put( 92,740){\circle*{2}}
\put( 92,780){\circle*{2}}
\put( 92,821){\circle*{2}}
\put(189,821){\circle*{2}}
\put(260,821){\circle*{2}}
\put(189,780){\circle*{2}}
\put(260,780){\circle*{2}}
\put(260,740){\circle*{2}}
\put(189,699){\circle*{2}}
\put(260,699){\circle*{2}}
\put(189,740){\circle*{2}}
\put( 21,820){\line( 0,-1){145}}
\put( 21,675){\line( 0, 1){145}}
\put( 21,780){\vector( 1, 0){ 71}}
\put( 92,780){\line(-1, 0){ 71}}
\put(260,821){\line( 0,-1){145}}
\put(260,676){\line( 0, 1){145}}
\put(189,821){\line( 0,-1){122}}
\put(189,699){\line( 0, 1){122}}
\put(260,821){\vector(-1, 0){ 71}}
\put(189,821){\line( 1, 0){ 71}}
\put(260,780){\vector(-1, 0){ 71}}
\put(189,780){\line( 1, 0){ 71}}
\put(260,740){\vector(-1, 0){ 71}}
\put(189,740){\line( 1, 0){ 71}}
\put(260,699){\vector(-1, 0){ 71}}
\put(189,699){\line( 1, 0){ 71}}
\multiput(100,780)(4.40000,0.00000){22}{\makebox(0.4444,0.6667){\small .}}
\multiput(100,740)(4.40000,0.00000){22}{\makebox(0.4444,0.6667){\small .}}
\put( 21,821){\vector( 1, 0){ 71}}
\put( 92,821){\line(-1, 0){ 71}}
\put( 95,821){\vector(-1, 0){ 2}}
\put( 95,780){\vector(-1, 0){ 2}}
\put( 95,740){\vector(-1, 0){ 2}}
\multiput(100,821)(4.40000,0.00000){22}{\makebox(0.4444,0.6667){\small .}}
\put( 92,821){\line( 0,-1){ 81}}
\put( 92,740){\line( 0, 1){ 81}}
\put( 21,740){\vector( 1, 0){ 71}}
\put( 92,740){\line(-1, 0){ 71}}
\put(  55,825){\makebox(0,0)[lb]{\raisebox{0pt}[0pt][0pt]{$\theta$ }}}
\put(  140,825){\makebox(0,0)[lb]{\raisebox{0pt}[0pt][0pt]{$\phi$ }}}
\put(  225,825){\makebox(0,0)[lb]{\raisebox{0pt}[0pt][0pt]{$\tau$ }}}
\put(  -5,768){\makebox(0,0)[lb]{\raisebox{0pt}[0pt][0pt]{\small
$P^{\theta^{-1}}$}}}
\put(-10,727){\makebox(0,0)[lb]{\raisebox{0pt}[0pt][0pt]{\small ker$\theta$ }}}
\put( 95,810){\makebox(0,0)[lb]{\raisebox{0pt}[0pt][0pt]{\small $K$}}}
\put( 95,768){\makebox(0,0)[lb]{\raisebox{0pt}[0pt][0pt]{\small $P$}}}
\put( 85,727){\makebox(0,0)[lb]{\raisebox{0pt}[0pt][0pt]{\small $\langle
1\rangle$}}}
\put(175,810){\makebox(0,0)[lb]{\raisebox{0pt}[0pt][0pt]{\small $H$}}}
\put(175,768){\makebox(0,0)[lb]{\raisebox{0pt}[0pt][0pt]{\small $\hat{P}$}}}
\put(175,727){\makebox(0,0)[lb]{\raisebox{0pt}[0pt][0pt]{\small $N$}}}
\put(180,685){\makebox(0,0)[lb]{\raisebox{0pt}[0pt][0pt]{\small $\langle
1\rangle$}}}
\put(265,810){\makebox(0,0)[lb]{\raisebox{0pt}[0pt][0pt]{\small $G$}}}
\put(265,768){\makebox(0,0)[lb]{\raisebox{0pt}[0pt][0pt]{\small 
$P^{\tau^{-1}}$}}}
\put(265,685){\makebox(0,0)[lb]{\raisebox{0pt}[0pt][0pt]{\small ker$\tau$}}}
\put(  5,810){\makebox(0,0)[lb]{\raisebox{0pt}[0pt][0pt]{\small $G$}}}
\end{picture}
\end{center}

The subgroup $N=\ker\phi$ plays a role similar to that of the subgroup
$M$ in the  $\L$-covering algorithm.  It is  the maximal \Fp$K$-module
by which $K$ can be extended so that $P$ acts trivially on $N$ and the
extension  is   an  epimorphic   image  of  $G.$   Thus   $N$  is   an
\Fp$(K/P)$-module.   If $P$ is non-trivial the extension of $K$ by $N$
has the same  generator  number as  $K,$ because,  by Burnside's Basis
Theorem (Huppert  I, Satz  3.15, 1967),  the generator number  of  the
extension of $P$ by  $N$ is already determined by the generator number
of  $P.$  If  $P$  is  non-trivial  the  module  $M$  is  the  largest
\Fp$K$-module by which $K$ can be extended such that the extension has
the  same  generator number  as  $K$  and  $P$ acts  trivially on $M.$
Therefore $N$ is a  factor module of $M.$ If $P$ is trivial and $N$ is
the largest \Fp$K$-module by which  $K$ can  be extended such that the
extension is a  homomorphic image of $G,$ it does not follow  that $N$
is isomorphic to a factor module  of $M,$ since the extension may have
a larger  generator number  than  $K.$  However,  in both cases we can
write down a finite presentation for the extension.

This presentation is obtained as follows.  Let $\pcpG$ be the supplied
consistent  power conjugate presentation for $K,$ where ${\mathcal A} = \{
a_1,\ldots, a_n \}$ and
$$\R = \{ a_i^{p_i}= v_{ii},\,
	a_k^{a_j}= v_{jk} \mid
1\le i\le n,\,1\le j<k\le n\}.$$
Let $\pcptG$  with $\tilde{\A} = \{ a_1, \ldots,  a_n, y_{1},  \ldots,
y_s \}$  be the finite presentation for $\hat{K}$ as calculated by the
$\L$-covering algorithm.   Then $G/\L_{i,j+1}(G)$  is  isomorphic to a
quotient  of $\hat{K},$ if $P$ is  nontrivial. If $P$ is trivial,  the
generator number of $G/\L_{i,j+1}(G)$ may be larger than the generator
number of  $G/\L_{i,j}(G).$ Since $G/\L_{i,j+1}(G)$  has a  consistent
power conjugate presentation refining its $\L$-series  it follows that
any   additional   generators   lie  in   $\L_{i,j}(G)/\L_{i,j+1}(G).$
Therefore $N$ is isomorphic to a quotient of the direct product $Z$ of
the  free \Fp$(K/P)$-module  $Y$ and the free \Fp$(K/P)$-module on the
additional generators.  Let $t$  be  the  number of generators  of $G$
whose images under $\theta$ are not definitions, then $t = b-d,$ where
$b$ is the number of generators of $G$ in the finite presentation  and
$d$  is the  generator number  of $K.$  Add  new generators  $\{  z_1,
\ldots, z_t \}$ to $\tilde{\A}.$ The  set of relations $\tilde{\R}$ is
modified in the following manner.
\begin{enumerate}
\renewcommand{\theenumi}{\arabic{enumi}}
\item  add to $\tilde{R}$ all relations of  the form $[  z_{i}^{},
z_{j}^g ]=1,$   $[y_k, z_j^g ] = 1$ and $[z_i, y_k^g ] = 1$
for all normal $g = w(a_1, \ldots, a_r)$  
for $ 1 \le i,j \le t$ and $1 \le k \le
m$ and  all relations $z_{i}^p  = 1$ for $1 \le i \le t;$
\item add to $\tilde{R}$ all  relations $z_{i}^{a_j}  = z_i^{}$
for  $j >  r$  for $1 \le i \le t.$
\end{enumerate}

The group $\tilde{K}$  defined by $\pcptG$ has $G/\L_{i,j+1}(G)$ as  a
factor group. It is called  the {\it extended $\L$-covering  group  of
$K$} and $\pcptG$ is  the  {\it extended $\L$-covering  presentation.}
Define  a map $\sigma$  from $\{ g_1,  \ldots, g_b  \}$ to  the  group
$\tilde{K}$  by $g_i^{\sigma} =  g_i^\theta z_{k}$  if $g_i^\theta$ is
non-defining   and  $g_i^\sigma  =  g_i^\theta$  if   $g_i^\theta$  is
defining. The map $\sigma$ is called the {\it extended\/} map.

The basic step is illustrated by  an example.  Consider  the group $G$
defined by the following finite  presentation $\{ x,  y \mid x^8, y^3,
(x^{-1}y)^2,   (yx^3yx)^2 =   x^4\}.$ Let   $\L$    be   the list   $[
(2,1),(3,1),(2,2)].$ Then it  can  be  shown that  $G/\L_{3,1}(G)$  is
isomorphic to    $S_4.$   A   labelled  consistent   power   conjugate
presentation for $S_4$ was given above.   The input for the basic step
is the finite   presentation for  $G,$ the  labelled  consistent power
conjugate  presentation  for $S_4$ and  the  epimorphism $\theta  :  G
\rightarrow G/\L_{3,1}(G)$ defined by $x \mapsto a$ and $y \mapsto b.$
The  images   of $\theta$  are    the definitions of   $a$  and   $b,$
respectively.    We  have previously   determined  a  presentation for
$\hat{S_4}.$   This  is also   the extended $\L$-covering presentation
since both images  of $\theta$ are  definitions.   The map $\sigma$ is
the map from $G$ to $\tilde{K}$ which maps $ x$ to $a$ and $y$ to $b.$
Using the map  $\sigma$ the kernel  of the homomorphism  from $Z$ onto
$N$ can be computed effectively.

\begin{theorem}
\label{computeN}
Let $T$ be the set of elements defined in Theorem \ref{consisrels}.
Let $U$ be the  set  $ \{ r_i( g_1^{\sigma}, \ldots, g_b^{\sigma})  \mid 1
\le i \le m \}$  of   elements   of $Z$  obtained  by evaluating  the
relators of $G$ in the images of the generators  of  $G$ under the map
$\sigma.$ Then $N$ is isomorphic to $Z/((T\cup U)$\Fp$(K/P)).$
\end{theorem}

\begin{proof}
Consider the factor group $H$ of $\tilde{K}$ obtained by extending the
group $K$  by  $Z/((T\cup  U)$\Fp$(K/P)).$  Then  $H$  is generated by
$g_1^\sigma,  \ldots, g_b^\sigma.$ Since the  relations of $G$ hold in
$H$ it follows that  it is a homomorphic image of  $G/\L_{i,j}(G).$ By
construction $H$ has $G/\L_{i,j}(G)$ as a homomorphic image, hence $H$
is isomorphic to $G/\L_{i,j}(G).$
\end{proof}

In our example  $U$ is the set 
$\{ y_1^b y_3^{(bab)} y_4^b y_5^b y_7^{(1 +  b)} \}.$

The vector enumerator was used to compute a vector space basis for the
module  $M$ in the $\L$-covering  algorithm.  Here it is  employed  to
compute  a  vector  space  basis   for  the  module  $N   =  Z/((T\cup
U)$\Fp$(K/P)).$ It takes as input
\begin{enumerate}
\renewcommand{\theenumi}{\arabic{enumi}}
\item a  consistent power-conjugate presentation  for $K;$ 
\item the  set of generators  for $Z;$ 
\item the  set of relations $T  \cup  U.$ 
\end{enumerate}

\noindent The output  is 
\begin{enumerate}
\renewcommand{\theenumi}{\arabic{enumi}}
\item an  \Fp-basis  for $N;$  
\item an expression in this basis for the  image  under the
generators of $K/P$ of every   basis element;   
\item expressions for  the   images  of  the
\Fp$(K/P)$-generators  of $Z$ in  terms of  the basis elements.    
\end{enumerate}

\noindent This output is used  to obtain  a consistent power conjugate
presentation for  the extension  $H$ of  $K$  by  $N,$ an  epimorphism
$\tau$ from $G$ to $H$ and an  epimorphism $\phi$ from $H$ to $K.$ The
method for constructing the consistent power conjugate presentation is
again the method described in  the  proof of Theorem \ref{consisrels}.
The homomorphism $\tau$ from $G$  to $H$  is obtained by replacing the
elements  $z_i$ in the map $\sigma$  by the corresponding word in  the
basis for $N.$  This yields an epimorphism by  Theorem \ref{computeN}.
As pointed out earlier a base change for the vector space basis of the
module $N$ may be  necessary  in order to obtain a labelled consistent
power conjugate  presentation  for  $H.$ A  base  change  may  also be
necessary in  order to transform  $\tau$ into a labelled homomorphism.
The map $\tau$  is an epimorphism, since all the generators of $H$ are
either  defined as images of the generators of  $G$ under $\tau$ or by
definitions in the power conjugate  presentation of  $H$ on the images
of those generators.

The  vector enumerator is employed to compute a vector space basis for
the module  $Z/((T\cup  U)$\F{2}$(K/P))$  in the previous example. The
vector enumerator  returns  the basis  $\{ e, f, g \}$ defined by $e =
y_3,$ $f =  y_4$ and $g = y_5.$ Again, the vector enumerator gives the
action of $a$ and $b$ on this basis, while $c$  and $d$ act trivially.
The  information  is   used   to  construct  for  the  quotient  $H  =
G/\L_{4,0}(G)$  the  following  labelled  consistent  power  conjugate
presentation:

{\small
$$\begin{tabular}{llllllll}
\multicolumn{8}{l}{$\{\, a,\, b,\, c,\, d,\, e,\, f,\, g\, \mid $}\\
&$a^2=: c,$\\
&$b^{a}=\hphantom{:} b^2 c,$&$ b^3,$\\
&$c^{a}=\hphantom{:}c,$&$ c^b=:d,$&$c^2=:e,$\\
&$d^{a}=:cdf,$&$d^b=:{cdg},$&${d}^{c}=\hphantom{:}dfg,$&$d^2= ef,$\\
&$e^{a}=\hphantom{:} e,$&${e}^{b}=\hphantom{:}{ef},$&
${e}^{c}=\hphantom{:}e,$&$e^d= e,$&$ e^2,$\\
&$f^{a}=\hphantom{:} eg,$&${f}^{b}=\hphantom{:}efg,$&
$f^c=\hphantom{:}f,$&$f^d=f,$& $f^e= f,$ & $f^2,$ \\
&$g^{a}=\hphantom{:} ef,$&$g^b=\hphantom{:}e,$&
$g^c=\hphantom{:}g,$&$g^d=g,$& $g^e =g,$ &$ g^f=g,$&$ g^2\,\}$\\
\end{tabular}
$$
}

\noindent  and   the  labelled  epimorphism   $\tau$  from  $G$  onto
$G/\L_{4,0}(G)$  defined by  $x \mapsto a$ and $y \mapsto b.$ Hence in
this example  $G$ has a homomorphic image isomorphic to a factor group
of $\hat{S_4}.$ In fact $G$ is an extension of $S_4$ by a group $N$ of
order $2^3.$ The group $N$ is generated by $\{ c^2, d^2, [d,c] \}$ and
therefore $\L^{-2}(H)$ is isomorphic to the  $2$-covering group of the
Klein-4 group.

\section{Practical aspects}

This section focuses  on practical aspects of our algorithm.  First we
highlight  some  features  of  an   implementation. 

\subsection{The implementation}

The soluble quotient algorithm  as  described  in  Section 5 has  been
implemented in {\bf C}  and this first implementation  is known as the
ANU Soluble Quotient  Program  (SQ).   It  uses Version 3  of Linton's
implementation of a vector  enumerator  (see Linton 1991 and 1993) and
Nickel's  parser for finite  presentations as  used  in  his Nilpotent
Quotient Program, (see Nickel 1992).

The  basic step of the algorithm first computes a finite  presentation
$\pcptG$  for  the  extended $\L$-covering group  $\tilde{K}$  of  $K$
(given by the power conjugate presentation $\pcpG$), where $\tilde{\A}
=  \A  \cup \{ y_1, \ldots, y_s \}  \cup \{  z_1, \ldots, z_t \}.$ The
number $s+t$ of extra generators for $\tilde{\A}$ has a serious impact
on the performance of the algorithm.  Related quotient algorithms face
the same problem. Methods to reduce the number  of extra generators in
the  context of a  $p$-quotient algorithm  and  a  nilpotent  quotient
algorithm are described in Celler et  al.\ (1993) and Nickel (1993 and
submitted), respectively.  Similar ideas  can be used in our algorithm
when repeatedly computing $\L$-covering groups for a fixed  prime $p.$
These features can be invoked by an option  given  to  the program (SQ
1).

Theorem  6 uses  the set of normal words $T\cup U.$ Recall that a word
in  $\tilde{\A}$  is  normal  if it  is  of  the form  $w(a_1, \ldots,
a_n)\cdot  \Pi_{i=1}^s  y_i^{f_i},$  where  $f_i$  is  an  element  of
\Fp$(K/P).$ We  implement  the steps  referred  to as  ``collection in
$\hat{K}$'' in  Section 4.2 following  the strategy of collection from
the left (see Leedham-Green and Soicher, 1990).  In representing $f_i$
one  may  write it  as a  sum over \Fp\ of either  arbitrary or normal
words in the generators of the consistent power conjugate presentation
of  $K/P.$  This  can  again be determined by  an option given to  the
program  (SQ  2).   The influence  of the  options is  demonstrated in
Section 6.4.

\subsection{The test examples}

We highlight  certain  aspects  of  the  behaviour  of  SQ  using  the
following example presentations.

\begin{center}
{\small
\renewcommand{\arraystretch}{1.2}
\begin{tabular}{ll}
\vphantom{$P^1_1$}$P_1$ &
$\{ a,b\mid (ab)^2b^{-6},\, a^4b^{-1}ab^{-9}a^{-1}b \},$ \\
\vphantom{$P^1_1$}$P_2 $ &
$\{ a,b\mid a^2ba^{-1}ba^{-1}b^{-1}ab^{-2},\, a^2b^{-1}ab^{-1}aba^{-1}b^2\},$\\
\vphantom{$P^1_1$}$P_3$ &
$\{ a,b\mid ab^2(ab^{-1})^2,\, (a^2b)^2a^{-1}ba^2(bab)^{-1} \},$\\
\vphantom{$P^1_1$}$P_4$ &
  $\{ a,b\mid ab^2a^{-1}b^{-1}ab^3,\, ba^2b^{-1}a^{-1}ba^3 \},$\\
\vphantom{$P^1_1$}$P_5$ &  $\{ a,b\mid ab^3a^{-1}b^{-1}ab^3,\,
ba^3b^{-1}a^{-1}ba^3\},$\hphantom{$abababababaabababababaaaaaaaaaaaaaaaaaa$}\\ 
\end{tabular}
}
\end{center}

\begin{center}
{\small
\renewcommand{\arraystretch}{1.2}
\begin{tabular}{ll}
\vphantom{$P^1_1$}$P_6$ & $\{ a,b\mid ab^2a^{-1}b^{-1}a^3b^{-1},\,
  ba^2b^{-1}a^{-1}b^3a^{-1}\},$ \\
\vphantom{$P^1_1$}$P_7$ &
$\{ a,b\mid (ab)^3b^6,\, aba^{-1}bab^3aba^{-1}bab^{-1},\,
		 a^2b^{-2}aba^{-1}b^3ab^{-1}a^{-1}b^2 \},$\\
\vphantom{$P^1_1$}$P_8$ & $\{ a, b, c, d, e \mid
a^{-1}cd^{-1}e(dc)^{-1}e(da)^{-1}ec^{-1}be(adc)^{-1}eb^{-1}c,$\\
&$\hphantom{\{}(adc)^{-1}\!e(ab)^{-1}\!c(be)^{-1}\! c d^{-1}\!e (adc)^{-1}\! e
(ab)^{-1}\! 
d^{-2} ea e^{-1}\!a dc e^{-1}\!d c^{-1}\!a e,$\\
&$\hphantom{\{}acdce^{-1}\!de^{-1}\!da(ae)^{-1}\!e(ea)^{-1}\!d^2ab
(ea)^{-1}\!d^2abe^{-1}\! 
adabe^{-1}\!adce^{-2}\!d, $\\
&$\hphantom{\{} dc(dabe)^{-1}\!ed^{-1}\!e(adc)^{-1}\!e(ea)^{-1}\!
d^2ace(adc)^{-1}\!e(ab)^{-1}\!c(be)^{-1}\!c d^{-1}\!ec^{-1}\!ae^{-1}\!a, $\\
&$\hphantom{\{} c^{-1}\!e(adc)^{-1}\!e(adab)^{-1}\!e(ea)^{-1}\!dacd^{-1}\!
ec^{-1}\!a bce^{-1}\!da(ae)^{-1}\!e(ea)^{-1}\!d^2abe^{-1}\!ad, $\\
&$\hphantom{\{} ce^{-1}\!de^{-1}\!da(ae)^{-1}\!e(ea)^{-1}\!d^2abe^{-1}\!
adabce^{-1}\!da(ae)^{-1}\! e(ea)^{-1}\!d^2abe^{-1}\!ad^2, $\\
&$\hphantom{\{} ae(adc)^{-1}\!e(adab)^{-1}\!e(ea)^{-1}\!da^2c^{-1}\!b
(eda)^{-1}\!da(ae)^{-1}\!e(ea)^{-1}\!d^2abe^{-1}\!adce^{-2}\!d, $\\
&$\hphantom{\{}(adc)^{-1}\!e(ea)^{-1}\!d^2\!ab(e^{-1}\!a)^2\!dabe^{-1}\!adc
(ce)^{-1}\!ae(adc)^{-1}\!e(ab)^{-1}\!c(be)^{-1}\!c(bd)^{-1}\!cd^{-1}\!e,$ \\
&$\hphantom{\{}(adc)^{-1}\!e(ea)^{-1}\!de(adc)^{-1}\!e(adab)^{-1}\!
ece^{-1}\!d[a^{-1}\!\!\!\!,e](ea)^{-1}\!d(dabe^{-1}\!a)^2\!dc
(be)^{-1}\!cd^{-1}\!e \},$\\
\vphantom{$P^1_1$}$P_9$ & $\{ a,b \mid a^3, b^3, (ab)^6, (a^{-1}b)^6 \},$\\
\vphantom{$P^1_1$}$P_{10}$ & $\{ a,b \mid a^3, b^6, (ab)^6, (a^{-1}b)^6 \},$\\
\vphantom{$P^1_1$}$P_{11}$ & $\{ a,b \mid a^6, b^6, (ab)^6, (a^{-1}b)^6\},$\\
\vphantom{$P^1_1$}$P_{12}$ & $\{  a,\, b,\, c \mid   a^2, b^{31}, c^5,$
$c^{-1}bcb^-2, c^{-1}acb^{-1}ab,abac^{-1}acab^{-1},$\\
&$\hphantom{\{} c^{-1}ac^2ab^2ac^{-1}b^{-1},$ $(b^{-1}ab^2ab^{-1})^2,$
$c^{-1}b^{-2}cab^3abbab^{-1} \},$ \\
$P_{13}$ &  $\{  a, b, c\mid 
a^2,  b^{127}\!,  c^7, c^{-1}bcb^{-2}, 
b^{-1}abc^{-1}bab^{-1}c, (a^ca)^2, (aa^{b^{-1}})^2,$
$    ac^{-1}acb^{-1}cac^{-1}b, $\\
&$\hphantom{\{}ab^{-1}c^{-1}b^{-1}ab^2cab^{-1},$
$(ac^{-1}b^{-1}abc)^2,$
$(bab^{-1}c^{-1}ac)^2,$
$(b^{-1}ab^2ab^{-1})^2\},$\\
\vphantom{$P^1_1$}$P_{14}$ & $\{  a,\, b,\, c,\, d,\, e,\, f \mid 
a^2,\, b^2,\, c^2,\, d^2,$ $ e^{15},\, f^4, $
$[b, a],\, [c, a],\, [d, a],\, [c, b],\, [d, b],\, [d, c],$\\
& $\hphantom{\{}  f^{-1}efe^{-2},$
$e^{-1}\!aeb^{-1},$
$e^{-1}\!bec^{-1},$
$e^{-1}\!ced^{-1},$
$e^{-1}\!deb^{-1}a^{-1}, $
$f^{-1}\!afa^{-1}, $
$f^{-1}\!bfc^{-1},$ \\
&$ \hphantom{\{} f^{-1}\!cfb^{-1}\!a^{-1},$
$f^{-1}\!dfd^{-1}\!c^{-1} \}.$\\
\end{tabular}
}
\end{center}

The  first seven presentations appear  in Wegner (1992); presentations
$P_1,$ $P_2,$ and $P_3$ were constructed  by Kenne  (1990) and  $P_4,$
$P_5$ and $P_6$ appear  in Campbell (1975).  Presentation  $P_8$ arose
during Neub\"user's study of the Heineken group (Neub\"user and Sidki,
1988).   Presentations  $P_9,$ $P_{10},$ and $P_{11}$ were constructed
in  the  study  of  $B(2,6)$  (see  Havas,  Newman  and  Niemeyer  (in
preparation)).   Presentations  $P_{12},$ $P_{13}$ and $P_{14}$ define
soluble groups of derived length 3.  They arise as semidirect products
of the multiplicative and additive groups of finite fields extended by
the Frobenius automorphism.  They have been suggested by  Pasechnik.

\subsection{The {$\L$}-series of the test examples}

An  important problem in applying the soluble quotient algorithm is to
choose an appropriate  $\L$-series to determine  a factor group  of  a
finitely  presented  group.  The  program Quotpic, by  Holt  and  Rees
(1992), has  proved to be  valuable in solving  this problem. It is an
interactive,  graphical tool  which, among other things,  incorporates
very  practical routines for manipulating  presentations  and provides
access to  various  other programs.  It  was  employed to compute  the
$\L$-series for some of the test examples.

One possible method of determining an $\L$-series for a finite soluble
quotient  of  a  finitely  presented group $G$  using Quotpic  is  the
following.  A prime $p$ in the list  of  abelian  invariants of $G$ is
chosen and the ANU $p$-Quotient  Program (see Newman and  O'Brien,  in
preparation) is called to compute a power conjugate presentation for a
$p$-quotient  $G/K$ of desired exponent $p$-class.  A presentation for
$K$  can  be  obtained by  Reidemeister-Schreier rewriting.   Now  the
process can be repeated with $K$ taking the role of $G.$

Obviously it is only possible to compute presentations  for the kernel
if the quotient of the  group  over the kernel is  small.  This is the
case  in presentations $P_8,$ $P_9,$ $P_{10}$ and $P_{11},$ where  the
3-quotient has a kernel of index at most $3^2$ in the group defined by
the presentation.  For presentations $P_{12},$  $P_{13}$  and $P_{14}$
the $\L$-series were known beforehand.

\subsection{Performance}

Table 1 gives the CPU time in seconds taken  on a DEC  3000/600 AXP
(with 96Mb RAM) to
compute  consistent  power   conjugate   presentations  for  the  test
examples.   For  each  presentation  we  list  the  $\L$-series   that
determines the finite soluble quotient computed and its order.

\begin{center}
{{\bf Table 1.} Quotient group specifications and run times}
{\small
\renewcommand{\tabcolsep}{4.0pt}
\begin{tabular}{rrlrrrrr}
\hline \hline
 $P$& Order & $\L$-series & SQ  &  SQ 1  & SQ 2 \\
\hline
1&$2^4\!\cdot\! 3^4$&{ [(2,1),(3,1),(2,2),(3,2)]}&
2.1&7.6&4.0\\
 2&$2^5\!\cdot\! 3\cdot\! 5^2 $&{ [(2,1),(3,1),(2,2),(5,1)]}&  
2.9&29.6&8.7\\
 3&$2^3\!\cdot\! 3\!\cdot\! 5^3$&{ [(3,1),(2,2),(5,2)]}& 
4.1&4.0&2.4\\
4&$2^8\!\cdot\! 3^3$&{ [(3,1),(2,2),(5,1),(11,1)]}& 
27.9&129.6&60.0\\
 5&$2^3\!\cdot\! 3^6$&{ [(2,3),(3,2)]}&
0.5&0.4&0.4\\
 6& $2^3\!\cdot\! 3\cdot\! 5\!\cdot\! 11$&{ [(3,2),(2,2)]}&
0.6&0.5&0.6\\
 7& $2^{9}\!\cdot\! 3^4$&{ [(3,1),(2,2),(3,2),(2,1)]}&
$+1220$&$+1260$&$+1411$\\
 8& $2^{26}\!\cdot\! 3$&{ [(3,1),(2,11)]}& 
988.3&348.3&1021.1\\
 9& $2^{252}\!\cdot\! 3^3$&{ [(3,2),(2,3)]}&
1453.5&347.7&2646.2\\
10&$2^{182}\!\cdot\! 3^3$&{ [(3,2),(2,2)]}&
77.8&58.3&81.6\\
 11&$2^{398}\!\cdot\! 3^3$&{ [(3,2),(2,2)]}&
473.1&336.6&675.2\\
 12& $2^5\!\cdot\! 5\!\cdot\! 31$&{ [(5,1),(31,1),(2,1)]}&
0.3&0.3&0.4\\
 13&$2^7\!\cdot\! 7\!\cdot\! 127$&{ [ (7,1), (127,1), (2,1) ]}&
31.5&31.4&36.4\\
14&$2^6\!\cdot\! 3\!\cdot\! 5$&{ [(2,2),(3,1),(5,1),(2,1),(3,1)]}&
0.6&0.6&0.7\\
\hline \hline
\end{tabular}
}\end{center}

Some entries contain the symbol '$+$'. In these cases the computations
were not  completed.  The vector enumerator invoked by SQ, SQ 1 and SQ
2 ran out of space after the time listed for presentation $P_7.$ 

Some remarks on the space requirements of SQ and the vector enumerator
are  in  order.  SQ  1  allocated  less  than  1  Mb  for  almost  all
presentations; it needed 1.19 Mb for $P_{10},$ 2.6 Mb for $P_{9},$ and
4.8 Mb  for $P_{11}.$ These  space requirements are for  the SQ 1 only
and do not list the requirements for the vector enumerator. The vector
enumerator, for example,  allocated 3.6 Mb for the computation of  the
module of dimension 163 in $P_{10}$ when called from SQ 1, 54.8 Mb for
the module of dimension 205 in $P_{9}$ when called from SQ 2, but only
9.7 Mb when called from SQ 1.

\subsection{Comparisons}

Let  us first compare the different  versions of our algorithm.  SQ  1
performs  best  when the  presentation describes  a  group with  large
$p$-quotients;  for example  for presentations $P_8,$ $P_9,$  $P_{10}$
and $P_{11}.$ If  the  order of  the computed  quotient  involves many
different primes,  and  each occurs  to a small power, SQ  appears  to
perform better; for example for  presentations $P_1,$ $P_2$ and $P_4.$
This shows that the number of extra generators has an influence on the
performance  of  the  vector  enumerator.   When repeatedly  computing
$\L$-covering  groups for a fixed prime the presentations generated by
SQ 1 seem  to  be more suitable for  the vector  enumerator than those
computed by SQ.  However, in examples where SQ 1 performs worse than SQ
it spends additional time in  the  vector  enumerator.  This indicates
that the time is not lost in the methods to reduce the number of extra
generators,  but that the  presentations computed  are  worse for  the
vector enumerator.

The performance of SQ 2 seems to lie between that of SQ and SQ 1.  An
interesting  feature is  that  where it  performs slower  than SQ  the
additional  time  is  spent  in  the  vector  enumerator  and  not  in
collection.  Thus handing the vector enumerator  normal words does not
necessary seem to be better.

If  one wants to  show that a finitely presented  group with a soluble
quotient $G$ does not have  a  larger soluble quotient as an extension
of $G$  by a  module $N$ over \Fp, then the theoretical  bound on  the
dimension  of  $N$ is  generally very  large.   Presentation  $P_{14}$
describes  a group  $G$ of order $2^6 \cdot 3 \cdot 5.$ Hence  it does
not describe an extension of $G$ by a module over \F{3}.  Here  SQ can
prove that this is the  case  in $0.6$ seconds.

\subsection{Conclusions}

The performance of the vector  enumerator  is critical for the soluble
quotient  algorithm presented  here.   As  pointed  out  earlier,  its
performance depends on the module presentations supplied as input.  It
seems promising to investigate  more  thoroughly  how  a  presentation
should be supplied  as  input.  Further  it seems promising to develop
methods  for  reducing   the  number   of  generators  in  the  module
presentations.   For example, Leedham-Green (1984) has suggested using
special  power conjugate presentations.  Recently  various  aspects of
these  power conjugate presentations were investigated  by Eick (1993)
and  Cannon  and  Leedham-Green   (in   preparation).   Special  power
conjugate  presentations  can  be  used to  reduce the  number of  new
generators if  the soluble quotient is calculated by computing maximal
nilpotent factors.  These suggestions need to be tested. In fact, when
extending a soluble group by  an elementary  abelian  $p$-group, it is
not necessary  to introduce new  generators  for those relations whose
left hand sides only involve generators  whose power relations involve
a  prime not equal to $p.$ In that  case,  however, one has to add new
generators to all images of  the epimorphism of the finitely presented
group  onto  the  soluble  quotient   so  far  computed  and  use  Fox
derivatives,  see  Leedham-Green  (1984).   The  method  he  describes
computes  a module larger  than  the  one  required.   In  fact,  when
extending a soluble group $K$ this method computes  a module which has
a direct summand  $\hbox{\Z}_p^{|K|-1},$  which  seems  difficult  for
large $K.$

Verifying that a certain quotient group  of a finitely presented group
is indeed the largest  finite soluble  quotient  seems  to  be  a very
difficult  problem in practice.  Leedham-Green (1984)  suggested using
an integer  vector  enumerator to  determine  the  primes in the  next
nilpotent  factor.   Linton  has  written  a  version  of  the  vector
enumerator  that   works  over  the  integers.   Plesken  (1987)  also
describes a method to determine the primes that can occur in a soluble
quotient  of a finitely  presented group. Again,  these ideas  need to
tested.

\subsection{The Code} 

The  ANU  Solvable Quotient  Program  (Version  1.0) is  available via
email from the author ({\tt alice@maths.uwa.edu.eu}),  by
anonymous  {\tt ftp}  from  {\tt  maths.anu.edu.au}  in the  directory
{\tt pub/SQ}, and as a share library with \GAP\ 3.4. It might also
become available as part of {\sc Magma.}

\medskip

\noindent
{\sc Acknowledgments}

I thank my PhD supervisor Dr M.F.\ Newman for his generous support and
assistance; I thank Dr L.G.  Kov{\'a}cs,  Dr  C.R.\  Leedham-Green, Dr
Werner Nickel and  Dr E.A.\ O'Brien  for many  encouraging discussions
and generous  help.  I  thank  Dmitrii  V.\ Pasechnik for many helpful
discussions on the last section.

I  acknowledge  the  support of an OPRSA and  an  ANU PhD  scholarship
during  which this  work  was carried  out.  Part of  the  writing was
supported by ARC Grant A69230241.

\end{document}